\documentclass[12point]{amsart}
 \usepackage{amssymb,exscale,bbm,mathrsfs}
  \usepackage[all]{xy}
 \newtheorem{theorem}{Theorem}[section]
 \newtheorem{corollary}[theorem]{Corollary}
 \newtheorem{lemma}[theorem]{Lemma}
 \newtheorem{proposition}[theorem]{Proposition}
 \newtheorem{remark}[theorem]{Remark}
 \newtheorem{example}[theorem]{Example}
 \newtheorem{examples}[theorem]{Examples}
 \theoremstyle{definition}
 \newtheorem{definition}[theorem]{Definition}
 \numberwithin{equation}{section}
\begin{document}
  \renewcommand {\Re}{\text{Re}}
  \renewcommand {\Im}{\text{Im}}
   
  \newcommand {\sfrac}[2] { {{}^{#1}\!\!/\!{}_{#2}}} 
  \newcommand {\pihalbe}{\sfrac{\pi}2}
  \newcommand {\onehalf}{\sfrac{1}2}
  \newcommand {\Borel}{{\mathcal Bo}}
  \newcommand {\Poisson}{{\mathscr P}}
  \newcommand {\CC}{\mathbb C}  
  \newcommand {\DD}{\mathbb D}  
  \newcommand {\EE}{\mathbb E}
  \newcommand {\NN}{\mathbb N}
  \newcommand {\PP}{\mathbb P} 
  \newcommand {\ZZ}{\mathbb Z}
  \newcommand {\RR}{\mathbb R}
  \newcommand {\cT}{\mathcal T}
  \newcommand {\sigmaA} {\sigma_A}
  \newcommand {\sigmaP} {\sigma_P}
  \newcommand {\sigmaR} {\sigma_R}
  \newcommand {\BOUNDED}{\mathcal B}
  \newcommand {\RANGE}{\mathcal R}
  \newcommand {\DOMAIN}{\mathcal D}
  \newcommand {\FOURIER}{\mathcal F}
  \newcommand {\LAPLACE}{\mathcal L}
  \newcommand {\eins} {\mathbbm 1}
  \newcommand {\al}{\alpha}
  \newcommand {\la}{\lambda}
  \newcommand {\eps}{\varepsilon}
  \newcommand {\Ga}{\Gamma}
  \newcommand {\ga}{\gamma}
  \newcommand {\om}{\omega}
  \newcommand {\Om}{\Omega}
  \newcommand {\Xm}{X_{-1}}
  \newcommand {\SECT}[1] {S_{{#1}}}
  \newcommand {\CSECT}[1] {\overline{S_{{#1}}}}
  \newcommand {\STRIP}[1] {{\rm St}_{{#1}}}
  \newcommand {\norm}[1] {\| #1 \|}  
                                     
  \newcommand {\abs}[1] {|#1|}
  \newcommand {\lrnorm}[1]{\left\| #1 \right\|}
  \newcommand {\bignorm}[1]{\bigl\| #1 \bigr\|}
  \newcommand {\Bignorm}[1]{\Bigl\| #1 \Bigr\|}
  \newcommand {\Biggnorm}[1]{\Biggl\| #1 \Biggr\|}
  \newcommand {\biggnorm}[1]{\biggl\| #1 \biggr\|}
  \newcommand {\Bigidual}[3] {\Bigl\langle #1, #2 \Bigr\rangle_{#3}}
  \newcommand {\bigidual}[3] {\bigl\langle #1, #2 \bigr\rangle_{#3}}
  \newcommand {\idual}[3] {\langle #1, #2 \rangle_{#3}}
  \newcommand {\Bigdual}[2] {\Bigidual{#1}{#2}{}}
  \newcommand {\bigdual}[2] {\bigidual{#1}{#2}{}}
  \newcommand {\dual}[2] {\idual{#1}{#2}{} }
  \newcommand {\SP}[2]{ [#1 | #2] }
  \newcommand {\weg} {\backslash}
  \newcommand {\suchthat}{:\;}
  \newcommand {\emb} {\hookrightarrow}
  \newcommand {\McIntosh} {Mc{\hspace*{0.2ex}}Intosh}
\renewcommand{\labelenumi} {(\alph{enumi})}    
\renewcommand{\labelenumii}{(\roman{enumii})}
\renewcommand{\theenumi} {(\alph{enumi})}      
\renewcommand{\theenumii}{(\roman{enumii})}    
\allowdisplaybreaks
\title[Exact observability, square functions and spectral theory]
 {Exact observability, square functions and spectral theory}
 \author[Bernhard H. Haak]{Bernhard H. Haak} \address{%
   Institut de Math\'ematiques de Bordeaux\\%
   Universit\'e Bordeaux 1\\351, cours de la Lib\'eration\\%
   33405 Talence CEDEX\\FRANCE}
 \email{bernhard.haak@math.u-bordeaux1.fr} \thanks{The first author
   was partially supported by the ANR project ANR-09-BLAN-0058-01}
 \author[El Maati Ouhabaz]{El Maati Ouhabaz} 
 \address{%
Institut de Math\'ematiques de Bordeaux\\%
Universit\'e Bordeaux 1\\351, cours de la Lib\'eration\\%
33405 Talence CEDEX\\FRANCE}
 \email{ElMaati.Ouhabaz@math.u-bordeaux1.fr}

\date{\today}
\begin{abstract}
In the first part of this article we introduce the notion of a
backward-forward conditioning (BFC) system that generalises the notion of
zero-class admissibiliy introduced in \cite{XuLiuYung}. We can show
that unless the spectum contains a halfplane, the BFC property occurs
only in siutations where the underlying semigroup extends to a group.
In a second part we present a sufficient condition for exact
observability in Banach spaces that is designed for
infinite-dimensional output spaces and general strongly continuous
semigroups. To obtain this we make use of certain weighted square
function estimates. Specialising to the Hilbert space situation we
obtain a result for contraction semigroups without an analyticity
condition on the semigroup.
\end{abstract}
\maketitle
\section{Introduction}
In this article we study exact observability of linear systems $(A,
C)$ on Banach spaces of the form
\[
\left\{
  \begin{array}{lcl}
    x'(t) + A x(t) &=& 0 \\
    x(0) &=& x_0 \\
    y(0) &=& C x(t)
  \end{array}
\right.
\]
We suppose throughout this article that $-A$ is the generator of a
strongly continuous semigroup $T(t)_{t\ge 0}$ on a Banach space
$X$. For details on semigroup theory used frequently in this article
we refer to e.g. to the textbooks
\cite{EngelNagel,HillePhillips:semigroups,Pazy}. Since we deal with
unbounded operators in general, we will note $\DOMAIN(A)$ the domain of
$A$ and $\RANGE(A)$ its range.
Let $Y$ be another Banach space and suppose that the observation
operator $C: \DOMAIN(A)  \to Y$ is  bounded and  linear when 
$\DOMAIN(A)$ is  endowed with the graph norm $\norm{ x }_{\DOMAIN(A)}
= \norm{x} + \norm{ A x}$.  Here we denote by $\norm{\dot}$ the norm of $X$.  
Since the observation operator $C$ is
generally unbounded, the concept of admissibility is introduced. It
means that the output $y$ of the system (usually measured in $L^2$
norm) depends continuously on the initial value $x_0$.
\begin{definition}  
 We say that $C$  is $L^2$-admissible in time
  $\tau>0$ (for $A$ or for $T(t)_{t\ge 0}$) if there exists a constant
  $M(\tau)>0$ such that
\[
      \sup_{ x\in \DOMAIN(A), \norm{x}=1} 
      \int_0^\tau \norm{CT(t) x}_Y^2\,dt =: M(\tau)^2 < \infty.
\]
\end{definition}  
\medskip
\begin{definition}  
We say that $C$ is exactly
$L^2$-observable for $A$ (or for $T(t)$)  in time $\eta>0$ if there
exists a constant  $m(\eta)>0$ such that
\[
      \inf_{ x\in \DOMAIN(A), \norm{x}=1} \int_0^\eta
      \norm{CT(t)x}_Y^2\,dt =: m(\eta)^2 > 0.
\]
\end{definition}
For more information the notion of admissible observation (or control)
operators we refer the reader to the overview article
\cite{JacobPartington:survey} or, both for admissibility and
observability issues to the recent book \cite{TucsnakWeiss} and
references therein.  We summarise some well-known facts and notations:
When there is no risk of confusion, 'admissible' means $L^2$
admissible in {\em some} finite time $\tau>0$ and 'exact observable'
means exactly $L^2$-observable for $A$ in {\em some} finite time
$\eta>0$. We say that $C$ is infinite-time admissible if $M(\infty) <
\infty$ and exactly observable  in infinite time if $m(\infty) > 0$.
Finite-time admissibility does not depend on the choice of $\tau>0$.
Nevertheless it turns out to be useful to study the (clearly
non-decreasing) functions $t \mapsto m(t)$ and $t\mapsto M(t)$.  In
the 'dual' situation of a control operator $B$ the quantity
$m(\eta)^{-1}$ is often referred to as {\em control cost} of a system.
We refer e.g. to \cite{Miller:Lebeau-Robbiano,
  Seidman:survey-control-cost, TenenbaumTucsnak:control-cost} and
references therein for more details.
 
Independence of the time $\tau>0$ of the notion of admissibility means
that a lack of admissibility expresses either by $M(\tau)=\infty$ for
all $\tau >0$ or by $M(\tau)<\infty$ for finite $\tau$ while
$M(\infty)=\infty$. On the other hand, a lack of exact observability
expresses by $m(\eta)=0$ for $0<\eta<\eta_0$ for $\eta_0 \in
(0,\infty]$. We remark that Example~\ref{ex:JPP} below satisfies
$m(\eta)=0$ for $0<\eta< 2$, $m(2)=1$ while $m(\eta)\to {+}\infty$ for
$\eta\to {+}\infty$.
\medskip
Since most parabolic equations like for example the heat equation are
not exactly observable unless very special observations are chosen
whereas exact observability appears frequently for hyperbolic systems
such as the wave equation, it appears natural to study necessary
spectral conditions of the generator $-A$ that make exact
observability possible or impossible.  In this direction we extend and
complete former results of \cite{XuLiuYung}. We introduce the notion
of backward-forward conditioning  \ref{eq:etoile}-systems.  These are
admissible and exactly observable systems for which $M(\eta) <
m(\tau)$ for some $\eta < \tau$. We analyse spectral properties of the
generator $-A$ of the semigroup of such systems. In particular, we
prove that the approximate point spectrum of $A$ is contained in a
vertical strip. Therefore, the boundary of the spectrum is also
contained in a strip. We prove in addition that if $(A, C)$ is an
admissible \ref{eq:etoile}-system such that the spectrum of $A$ does
not contain a half-plane then the semigroup actually extends to a
group.  Note that every bounded group with an admissible operator $C$
is a \ref{eq:etoile}-system.  Since (\ref{eq:etoile}) is a frequent
property that typically is more likely to hold the more 'regular' the
operators $A$ and $C$ are, this shows that exact observability is
considerably rare outside the group context.
A second part of this paper is devoted to a new sufficient criterion
for exact observability.  Under an assumption of square function type
estimate we prove that a condition like
\[ 
  \norm{CA^{-\alpha} x}_Y \ge \delta  \norm{x},
\]
implies exact observability.  Here $\alpha \in (0,1)$ 
and $\delta$ is a positive constant.\\
Without any further assumption, we show that if $-A$ is the generator
of a contraction semigroup on a Hilbert space and $C: \DOMAIN(A) \to
Y$ is such that $ \norm{CA^{-\onehalf} x}_Y \ge \delta \norm{x}$, then
$(A,C)$ is exactly observable.
In order to state and prove our criterion we make a heavy use of square function estimate of type
\[ 
       \norm{x}^2  
 \le K^2 \int_0^\infty \norm{ (tA)^{-\beta} (T(2t^{2\beta})-T(t^{2\beta})) x }^2 
  \,\tfrac{dt}t,
\] 
where $K$ is a positive
constant, $\beta \in (0,1)$ and $T(t)$ denotes the semigroup generates by $-A$.  In the case  where $\beta=\onehalf$,  this corresponds to
a lower square function estimate 
\[ \norm{x}^2  
\le K^2 \int_0^\infty \norm{ \varphi(tA)x }^2   \,\tfrac{dt}t,
\]
where $\varphi(z) := z^{-\onehalf} (e^{-2z} - e^{-z} )$. On Hilbert spaces, it  is well known that such estimate is related to the 
holomorphic functional calculus of the operator $A$.  The needed results on this 
functional calculus and associated square function estimates for
sectorial operators will be sketched in the last two sections.
As we will explain later our criterion applies for bounded analytic
semigroups on Hilbert spaces whose generator admits a bounded
$H^\infty$-calculus -- but the first part of this paper reveals this
to be impossible for a large class of systems unless $A$ is
bounded. One important aspect of the criterion that might also help in
other situations is therefore how to avoid making use of analyticity
assumption of the semigroup.  We discuss  at the end of this papers two  examples. 
\section{BFC-systems}
Let $X$ and $Y$ be Banach spaces with norms $\norm{ \cdot }$ and
$\norm{ \cdot }_Y$, respectively. Throughout this section,
$(T(t))_{t\ge0}$ is a strongly continuous semigroup on $X$ whose
generator is denoted by $-A$.
\begin{definition}  
  An admissible observation operator $C$ for  $A$ is called {\em
    zero-class admissible}, if $\lim_{\tau\to0+} M(\tau) = 0$.
\end{definition}
Note that if the semigroup $T(t)$ is bounded analytic with generator
$-A$ then for all $\al \in [0,\onehalf)$, $A^\al$ is zero-class
admissible. This follows from the inequality $\norm{ A^\al T(t) } \le
M t^{-\al}$. Consequently, if $C$ is bounded on such a fractional
domain space $\DOMAIN(A^\al)$ and thus
\[
     \norm{ C T(t) x}_Y  
  \le M \left[ \norm{A^\al T(t) x} + \norm{T(t) x}  \right], 
\]
$C$ is zero-class admissible.\\
It is also a obvious  fact that every bounded operator $C: X \to Y$ is
zero-class admissible. Here $(T(t))_{t\ge0}$ is merely a strongly
continuous semigroup on $X$.
\medskip
Consider now the linear operator $\widetilde \Psi_\tau: X \to L^2(0,
\tau; Y)$ defined by $\widetilde \Psi_\tau x = CT(\cdot)x$.  Then
admissibility (i.e., $M(\tau)<\infty$) means that $\widetilde
\Psi_\tau$ is a bounded operator. If in addition $m(\tau) >0$, then
$\widetilde \Psi_\tau$ is injective and has closed range.  Therefore,
we may consider the operator $\Psi_\tau: X \to \RANGE(\widetilde
\Psi_\tau)$, $\Psi_\tau = \widetilde \Psi_\tau$.  We have
\begin{equation}\label{eq:norm-psi-tau-inv}
\frac1{m(\tau)} 
= \sup_{ x\in \DOMAIN(A), \norm{x}=1} \frac{\norm{x}}{\norm{\Psi_\tau x}} 
= \sup_{ x\in \DOMAIN(A), \norm{x}\not=0} \frac{\norm{x}}{\norm{\Psi_\tau x}} 
= \norm{\Psi_\tau^{-1} }.
\end{equation}
We introduce the following definition. 
\begin{definition}
  We say that the system $(A, C)$ has the backward-forward
  conditioning property or shortly that $(A, C)$ is
  a~\ref{eq:etoile}-system if there exists some $0< \eta < \tau $ such
  that $C$ is admissible and exactly observable in time $\tau$ and if
\begin{equation}  \label{eq:etoile}\tag{BFC}
  \norm{ \Psi_{\tau}^{-1} } \, \norm{\Psi_\eta} < 1.
\end{equation}
\end{definition}
The condition (\ref{eq:etoile}) is clearly a conditioning property
for the output operator with different times $\eta$ and $\tau$ which
correspond to a backward and forward evolution of the system.  It also
follows from (\ref{eq:norm-psi-tau-inv}) that (\ref{eq:etoile}) is
equivalent to
  \begin{equation} \label{eq:mM}
  M(\eta) < m(\tau) \qquad {\rm for\  some}\  \eta<\tau.
  \end{equation}
  Therefore, if $C$ is exactly observable in some time $\tau$ and of
  zero-class, then (\ref{eq:mM}) holds trivially by letting $\eta$
  sufficiently small. Hence, the system is \ref{eq:etoile}.  If $C$ is
  admissible at any $\tau > 0$ and if $m(t) \to +\infty$ for $t\to
  +\infty$, then (\ref{eq:mM}) holds and again the system is
 \ref{eq:etoile}.
    
  Zero-class admissible operators are introduced and studied in
  \cite{XuLiuYung}. See also \cite{JacobPartingtonPott:zero-class}
  from which we borrow a concrete example leading to an
 \ref{eq:etoile}-system in which $C$ is not zero-class. 
\begin{example}\label{ex:JPP}
The following example is taken from Jacob, Partington and Pott
\cite[Example 3.9]{JacobPartingtonPott:zero-class}. We shall use
Ingham inequalities to prove that our system is
 \ref{eq:etoile}. Similar ideas could be used in a more general class
of examples.  Consider an undamped wave equation on $[0,1]$ with
Dirichlet boundary conditions and Neumann type observation of the form
 \[  \left\{
    \begin{array}{ll}
      \tfrac{\partial^2}{\partial t^2} z(x, t) =
      \tfrac{\partial^2}{\partial x^2} z(x, t) & \text{for }  x\in
      (0,1), t\ge 0\\
      z(0, t) = z(1, t) = 0 & \text{for } t\ge 0 \\
      z(x, 0) = z_0(x) \quad\text{and}\quad \tfrac\partial{\partial t}
      z(x, 0) = z_1(x) & \text{for }x\in (0,1)\\
      y(t) = \tfrac\partial{\partial x} z(0, t)
    \end{array}\right.
    \]
We rewrite the system as a first order Cauchy problem
 \[  \left\{
    \begin{array}{ll}
      \tfrac{\partial}{\partial t} U = - A U(t), & t \ge 0\\
      U(0) = (z_0, z_1)\\
      C U(x,t) = \tfrac\partial{\partial x} f(0)
      \end{array}\right.
      \]
      where $A = \left( \begin{array}{cc} 0 & - I \\
          -\tfrac{\partial^2}{\partial x^2} & 0 \end{array} \right)$
      and $U = (f, g)$. The latter Cauchy problem is considered on the
      Hilbert space $H = H_0^1(0,1) \times L^2(0,1)$ endowed with the
      norm $\norm{ (f,g) } = \sqrt{\int_0^1 | f' |^2 dx + \int_0^1 | g
        |^2 dx}$. Note that by the Poincar\'e inequality,
      $\sqrt{\int_0^1 | f' |^2 dx}$ defines a norm on $H_0^1(0,1)$
      which is equivalent to the usual one.
      
      It is a standard fact that $-A$ generates a strongly continuous
      semigroup $T(t)$ on $H$. It is easy to see that $A$ has compact
      resolvent and the eigenvalues are $\la_n = - i n \pi, \; n\in
      \ZZ \setminus \{ 0\}$ with normalised eigenfunctions $U_n(x) =
      \left( \frac{\sin(n\pi x)}{in \pi}, \sin(n\pi x) \right)$ which
      form an orthonormal basis of $H$.  Fix $(f,g) \in H$ and denote
      by $\al_n = \langle (f,g), U_n\rangle_H$ (the scalar product in
      $H$). Then
\[
   \norm{ (f,g) }^2 = \sum_{n \in \ZZ,n\not= 0} | \al_n |^2
\] 
and 
\[
  CT(t)(f,g) = \sum_n \al_n e^{in\pi t} CU_n = -i \sum_n \al_n
  e^{in\pi t}.
\]
Using the well known Ingham inequalities (see e.g. \cite[p.
162]{Young:nonharmonic} or \cite[Theorem 4.3]{KomornikLoreti}) we
obtain the following estimates for all $\tau > 2$,
\[ 
    m(\tau)^2 \sum_{n} |\al_n|^2
\le \int_0^\tau  \biggl| \sum_{n} \al_n e^{in\pi t} \biggr|^2\,dt
\le M(\tau)^2 \sum_{n} |\al_n|^2
\]
with 
\[
   m(\tau)^2 \ge  \frac{2 \tau}{\pi} \left( 1 - \frac{4}{\tau^2}
   \right), \, M(\tau)^2 \le  \frac{8 \tau}{\pi} \left( 1 +
     \frac{4}{\tau^2} \right).
\]
This shows that $C$ is admissible at any time $\tau > 0$ and exactly
observable in time $\tau > 2$ with constant $m(\tau) \to + \infty$ as
$\tau \to + \infty$. This shows (\ref{eq:mM}) and hence
the system $(A,C)$ is backward-forward conditioning. \\
In order to see that $C$ is not zero-class, we consider small $\tau >
0$ and $f \in H_0^1(0,1)$ with Fourier coefficients $\al_n$ and note
that
\[ 
   \int_0^\tau  \biggl| \sum_{n} \al_n e^{in\pi t} \biggr|^2\,dt = \norm{
   \chi_{[0,\tau]} f }_2^2,
\]
where $ \chi_{[0,\tau]}$ denote the indicator function of
$[0,\tau]$. From this equality it is clear that
\[ 
   1 = \sup_{ \norm{ f }_2 = 1} \norm{ \chi_{[0,\tau]} f }_2^2 =  \sup_{ \norm{
     f }_2 = 1}  \int_0^\tau  \biggl| \sum_{n} \al_n e^{in\pi t}
   \biggr|^2\,dt = M(\tau)^2.
\]
Therefore, the right hand side does not converge to $0$ as $\tau \to 0$. 
\end{example} 
\begin{remark}
  The above example is a special case of the following situation: let
  $C$ be  admissible in some arbitrary time $\tau>0$ and exactly
  observable in some time $\eta>0$ for a {\em group}
  $U(t)_{t\in\RR}$. Observe that $\norm{ x } = \norm{U(-t) U(t)x} \le
  \norm{U(-t)}\norm{U(t)x}$ whence $\norm{ U(t)x} \ge \norm{ U(-t) 
  }^{-1} \norm{x}$. From
\begin{align*}
  \int_0^{n\eta} \norm{ CU(t)x}^2 \,dt
  & =   \sum_{j=0}^{n-1} \int_0^\eta \norm{ CU(t) U(j\eta)x}^2\,dt \\
  & \ge m(\eta)^2 \biggl(\sum_{j=0}^{n-1} \norm{ U(-j\eta)
  }^{-2}\biggr) \norm{x}^2,
\end{align*}
we then  infer $m(n \eta)\to +\infty$ for $ n \to+\infty$ whenever the sum
in the last expression diverges. This is in particular the case for
bounded groups $U(t)_{t\in\RR}$. By the admissibility of the system
$(A, C)$ one then obtains (\ref{eq:etoile}) by letting $n$
sufficiently large. 
\end{remark}
We thank Hans Zwart for pointing out this remark to us.
\section{Spectral properties of \ref{eq:etoile}-systems }
We consider the same notation $X$, $Y$, $A$, $(T(t))_{t\ge0}$ and $C:
D(A) \to Y$ as in the previous section.  Or aim here is to study
spectral properties of \ref{eq:etoile}-systems.  We will extend some
results which have been proved in \cite{XuLiuYung} in the context of
zero-class operators. We note also that related ideas and results were
obtained previously by Nikolski \cite{Nikolski:policopie-controle} in
the particular case of bounded observation operators $C$ on $X$.
Let us introduce the classical function  $\eps: \RR+ \to \RR^+$
defined by
\[
   \eps(t)  := \inf_{\norm{x}=1}\norm{T(t)x}.
\]
It is clear that $\eps(t)$ is strictly positive for all $t>0$ if this
holds for a single $t_0>0$. Indeed, from
\[
\norm{T(s)} \, \norm{T(t)x} \ge \norm{ T(t+s) x} \ge \eps(t) \norm{
  T(s)x } \ge \eps(t)\eps(s) \norm{x}
\]
one infers that 
\[
\eps(t) \norm{T(s)} \ge \eps(t+s) \ge \eps(t)\eps(s),
\]
for all $t, s \ge 0$. 
For this reason we distinguish the cases that $\eps(t)$
is strictly positive for all $t>0$ of that it vanishes for all $t>0$
and we note this by $\eps(t) >0$ or $\eps(t)=0$ respectively.
\medskip
The following lemma is essentially contained in
\cite{Nikolski:policopie-controle} and \cite{XuLiuYung}.
\begin{lemma}\label{lem:etoile-et-eps}
  If $(A, C)$ is an admissible and exactly
  observable \ref{eq:etoile}-system, then $\eps(t)>0$.
\end{lemma}
\begin{proof} By the definition of \ref{eq:etoile}-system, there exist
  $0 < \eta < \tau$ such that
\[ \delta := m(\tau)^2 - M(\eta)^2 > 0.\]
By the semigroup property, 
\begin{align*}
m(\tau)^2 \norm{x}^2 
& \le \int_0^\tau \norm{ CT(t)x}_Y^2\, dt\\
& = \int_0^\eta \norm{ CT(t)x}_Y^2\, dt 
  + \int_0^{\tau-\eta} \norm{ CT(t)T(\eta)x}_Y^2\, dt\\
& \le M(\eta)^2 \norm{x}^2 +  M(\tau{-}\eta)^2 \norm{T(\eta) x}^2
\end{align*}
which immediately yields 
\[
\norm{T(\eta)x}^2 \ge M(\tau{-}\eta)^{-2}(m(\tau)^2-M(\eta)^2)
\norm{x}^2\ge M(\tau{-}\eta)^{-2} \delta \norm{x}^2.
\]
Therefore, $\eps(\eta) > 0$, and hence $\eps(t) > 0$ for all $t > 0$. 
\end{proof}
\begin{lemma}
   Suppose that  $(A, C)$ is an admissible and exactly observable
 \ref{eq:etoile}-system. Then $T(t)^*$ is injective for one (and thus
  all) $t>0$ if and only if $T(t)$ extends to a group on $X$.
\end{lemma}
\begin{proof}
  We know by Lemma~\ref{lem:etoile-et-eps} that $\eps(t)>0$. This
  implies that $T(t)$ is injective and has closed image for all $t\ge
  0$. Thus, $T(t)$ is bijective if and only if $T(t)^*$ is
  injective. The latter is clearly independent of $t>0$ by the
  semigroup law. Indeed, if $T(t_0)^*$ is injective for some $t_0 >
  0$, so are all $T(s)^*$ for $s<t_0$ since $T(t_0)^* = T(t_0{-}s)^*
  T(s)^*$. If $s > t_0$ then we find $n\in \NN$, $\delta \in [0, t_0[$
  such that $T(s)^* = (T(t_0)^*)^n T(\delta)^*$, and the injectivity
  of $T(s)^*$ follows from that of $T(t_0)^*$ and $T(\delta)^*$.  We
  saw that $T(t)^*$ is injective for one (and thus all) $t>0$ if and
  only $T(t)$ is bijective which in turn by
\[ 
  S(t) := \left\{
    \begin{array}{lcl}
      T(t) & \text{if} & t\ge 0 \\
      T(t)^{-1} & \text{if} & t< 0
    \end{array}\right.
\]
  is equivalent to a group extension of $T(t)$ on $X$.
\end{proof}
For a closed operator $S$ on $X$ recall the notions of point spectrum
\[
\sigmaP(S) = \bigl\{ \la \in \CC \suchthat ker(\la I - S) \not= \{0\}
\bigr\},
\]
the approximate point spectrum 
\[ 
\sigmaA(S) = \bigl\{  \la \in \CC \suchthat \inf_{x\in  \DOMAIN(S),
  \norm{x}=1} \norm{\la x - Sx} = 0 \bigr\}
\]
and the residual  spectrum 
\[
\sigmaR(S) = \bigl\{  \la \in \CC \suchthat \text{range}(\la-S)
\text{ is not dense in } X \bigr\}.
\]
It is easy to see that $\sigmaR(S) = \sigma(S) \backslash
\sigma_A(S)$. Of course, $\sigmaP(S) \subseteq \sigmaA(S)$.
\begin{proposition}\label{prop:prap}
  Let $(A,C)$ be an admissible and exactly observable
  \ref{eq:etoile}-system. Then there exist no approximate point
  spectrum of $A$ with arbitrary large real parts.
  In particular, if $C$ is an admissible zero-class operator and $A$
  has a sequence of approximate point spectrum with arbitrary large
  real parts then $C$ is not exactly observable.
\end{proposition}
\begin{proof}
  Recall that $\exp( - t \sigmaA(A) ) \subseteq \sigmaA(T(t))$ (see
  \cite[p. 276]{EngelNagel}, note that $-A$ is the generator). If we
  find a sequence $\la_n \in \sigmaA(A)$ with $\Re(\la_n) \to +
  \infty$, then $e^{-t \la_n} \in \sigmaA(T(t))$.  Hence,
  $\inf_{\norm{x} = 1} \norm{ T(t)x - e^{-t \la_n} x } = 0$. By
\[
  \eps(t) \norm{x} \le \norm{ T(t)x} 
\le \norm{ T(t)x - e^{-t \la_n} x } + e^{-t \Re(\la_n)} \norm{x}
\]
we get $\eps(t) = 0$ which is incompatible with exact observability by
Lemma~\ref{lem:etoile-et-eps}.
\end{proof}
Since $-A$ is the generator of a strongly continuous semigroup,
$\sigma(A)$ is contained in a right-half plane.  On the other hand, it
is well known that the boundary of the spectrum $\partial \sigma(A)$
is contained in $\sigmaA(A)$.  We obtain from the previous proposition
that $\Re (\partial \sigma(A))$ is bounded, i.e., $\partial \sigma(A)$ is
contained in a vertical strip. Thus 
\begin{corollary}\label{cor:approx-point-spec}
 Let $(A,C)$ be an admissible \ref{eq:etoile}-system.  Then 
 \[
    \Re (\partial \sigma(A)) := \{ \Re (\la) \suchthat  \la
    \in \partial \sigma(A) \}
\]
 is bounded.
\end{corollary}
The following lemma is  known. 
\begin{lemma}\label{lem:well-known}
  Let $S \in \BOUNDED(X)$ satisfy $\norm{S x} \ge \ga \norm{x}$ for
  some $\ga>0$ and all $x\in X$ and assume $0 \in \sigma(S)$. Then
  there exists $\delta > 0$ such that $B(0, \delta) \subseteq
  \sigmaR(S)$.
\end{lemma}
The main result in \cite{XuLiuYung} which states that if $C$ is
zero-class admissible and $\sigmaR(A)$ is empty, then $T(t)$ extends
to a group. The next propositions  extend  this result.
\begin{proposition}\label{prop:commeXLY} 
  Let $(A,C)$ be an admissible \ref{eq:etoile}-system.  If $\Re
  (\sigmaR(A)) := \{ \Re \lambda \suchthat \lambda \in \sigmaR(A) \}$
  is bounded, then $(T(t))_{t\ge0}$ extends to a group on $X$.
\end{proposition}
\begin{proof}
  We know by Lemma~\ref{lem:etoile-et-eps} that $\eps(t) >0$. If
  $T(t)$ was not boundedly invertible for some $t>0$, then $0 \in
  \sigma(T(t))$. By Lemma~\ref{lem:well-known}, there exists $\delta_t
  > 0$ such that $B(0, \delta_t) \subseteq \sigmaR(T(t))$. Since
  $\sigmaR( T(t)) \backslash \{0\} = \exp( - t \sigmaR(A) )$ (see
  \cite[p. 276]{EngelNagel}) we obtain
    
\[
B(0, \delta_t) \backslash \{0\} \subseteq  \exp( - t \sigmaR(A) ),
\]
Therefore, there exists a real sequence $(\la_n) \in \sigmaR(A)$ such
that $\Re \la_n \to + \infty$. This contradicts the assumption.
\end{proof}
\begin{proposition} \label{prop:prohalf} Assume that $(A,C)$ is an
  admissible \ref{eq:etoile}-system.  If $\sigma(A)$ does not contain
  a half-plane then $(T(t))_{t\ge0}$ extends to a group on $X$.
\end{proposition}
\begin{proof} By Corollary~\ref{cor:approx-point-spec} we see that
  $\sigma(A)$ is either contained in vertical strip or contains a
  half-plane. Now we apply Proposition~\ref{prop:commeXLY} to
  conclude.
\end{proof}
Considering the right shift semigroup $T(t)$  on $L^2(\RR_+)$ with the
identity observation $C=\text{Id}$ provides an example of a
\ref{eq:etoile}-system (even a zero-class admissible one, see
\cite[Remark 3.1]{XuLiuYung}) for which no group extension is
possible. In this example, it is not difficult to check the spectrum of $A$ 
satisfies $\sigma(A) = \sigmaR(A) = \CC^+$ (the right half-plane). 
 This shows that the spectral condition in
Proposition~\ref{prop:prohalf} cannot be omitted. 
\smallskip
Assume that $(A,C)$ is an admissible \ref{eq:etoile}-system. If $T(t)$
is analytic, differentiable or merely eventually continuous then by
\cite[p. 113]{EngelNagel} $\sigma(A)$ does not contain a
half-plane. Thus, we conclude by Proposition~\ref{prop:prohalf} that
$(T(t))_{t\ge0}$ extends to group on $X$. 
\begin{proposition}
  Assume that $(A,C)$ is an admissible \ref{eq:etoile}-system. If
  $T(t)$ is compact for some $t > 0$, then $X$ has finite dimension.
\end{proposition}
\begin{proof}
  If $T(t)$ is compact for some $t > 0$ then $\sigma(A) = \sigmaP(A)$
  is discrete. It follows from Proposition~\ref{prop:prap} that
  $\sigma(A)$ is bounded. We conclude by
  Proposition~\ref{prop:commeXLY} that $(T(t))_{t\ge 0}$ extends to a
  group on $X$. Thus, $I = T(t) T(-t) $ is compact on $X$ and
  therefore $X$ has finite dimension.
\end{proof}
\section{Sufficient conditions for exact observability}
Our aim in this section is to derive conditions on $C$ and $A$ which
imply exact observability.  Our condition reads as follows
\begin{equation}\label{eq:Rbeta}
\norm{CA^{-(1- \beta)} x}_Y \ge \delta  \norm{x}
\end{equation}
for all $x\in \DOMAIN(A) \cap \RANGE(A)$. Here
$\beta \in (0,1) $ and $\delta > 0$ are constants.  Since we shall
assume that $A$ is injective, it may be  convenient to understand
(\ref{eq:Rbeta}) in the sense
\[ \norm{CA^{-1} x}_Y \ge \delta \norm{ A^{-\beta}x}.\] Of course,
$\RANGE(A) \cap \RANGE(A^\beta) = \RANGE(A)$. 
In the sequel we need some basic properties of the $H^\infty$
functional calculus for sectorial operators. This functional calculus
goes back to the work of \McIntosh \cite{McIntosh:H-infty-calc}.  More
recent publications of the meanwhile rich theory can be found in
\cite{Haase:Buch} or \cite{KunstmannWeis:Levico} and the references
given therein. We briefly sketch the needed results and definitions.
\begin{definition}
  We denote by $\SECT{\omega}$ the open sector $\{ z\in \CC^*
  \suchthat |\arg(z)|<\omega\}$ and by $\CSECT{\omega}$ the closure of
  $\SECT{\om}$ in $\CC$.  
  We call a closed operator $A$ on $X$ {\em sectorial of angle
    $\omega$} if $A$ is densely defined having its spectrum in
  $\CSECT{\omega}$ such that $\la R(\la, A) := \la (\la{-}A)^{-1}$
  of $A$ is uniformly bounded on the complement of each  strictly larger
  sectors $\SECT{\theta}$, $\theta>\omega$.
\end{definition}
Notice that if $-A$ generates a bounded semigroup $T(t)_{t\ge0}$, then
$A$ is sectorial of angle $\pihalbe$ by the Hille-Yosida
theorem. Moreover, the semigroup is (bounded) analytic if and only if
$A$ is sectorial of angle $<\pihalbe$.
\smallskip
Let $H^\infty(\SECT{\omega})$ denote the holomorphic and bounded
functions on $\SECT{\omega}$ and Let $H^\infty(\CSECT{\omega})$ denote
the holomorphic and bounded functions on $\SECT{\omega}$ that are
continuous and bounded on $\CSECT{\omega}$. We further consider the
ideal $H_0^\infty(\SECT{\omega})$ (respectively
$H_0^\infty(\CSECT{\omega})$) of all functions $f \in
H^\infty(\SECT{\omega})$ (respectively $H_0^\infty(\CSECT{\omega})$)
that allow an estimate $|f(z)| \le M \max(|z|^\eps, |z|^{-\eps})$.
The class of $H_0^\infty(\SECT{\omega})$ functions admits a natural
functional calculus for sectorial operators $A$ of angle $\om$.
Indeed, if $f \in H_0^\infty(\SECT{\theta})$ for some $\theta>\om$ and
if $\Ga = \partial \SECT{\theta}$ denotes the orientated path with
strictly decreasing imaginary part, the Cauchy integral
\[
   f(A) := \frac1 {2\pi i} \int_\Ga f(\la)\, R(\la, A)\,d\la
\]
converges absolutely in norm and defines therefore a bounded operator
$f(A)$.  If $A$ has, say, dense range, one obtains then a functional
calculus for all functions $f \in H_0^\infty(\SECT{\theta})$. 
\begin{definition} 
  We say that $A$ admits a {\em bounded $H^\infty(\SECT{\om})$
  (respectively $H^\infty(\CSECT{\om})${\rm )} functional calculus}
  if $f(A)$ is a bounded operator on $X$ and there exists a constant
  $M$ such that
  \begin{equation}  \label{eq:bounded-H-infty-calc}
  \norm{ f(A) } \le M \norm{f}_\infty
\end{equation}
  for all $f \in H^\infty(\SECT{\theta})$ and
  $\theta>\om$ (respectively for all $f\in
  H^\infty(\CSECT{\om})${\rm )}.
\end{definition}
Let us mention that by an approximation argument, if
(\ref{eq:bounded-H-infty-calc}) holds for all $f \in
H_0^\infty(\SECT{\theta})$ then it holds for all $f \in
H^\infty(\SECT{\theta})$. Therefore, it is enough to check the
validity of (\ref{eq:bounded-H-infty-calc}) for all $f \in
H_0^\infty(\SECT{\theta})$ to obtain bounded $H^\infty(\SECT{\om})$
functional calculus. Similarly, if (\ref{eq:bounded-H-infty-calc})
holds for for all $ f \in H_0^\infty(\CSECT{\om})$ we obtain a
$H^\infty(\CSECT{\om})$ functional calculus.
\medskip
We say that $A$ admits upper square function estimates on
$\SECT{\theta}$ if there is a constant $M>0$ such that
  \begin{equation}\label{eq:abstract-upper-eq-f-est}
     \forall x \in X:\qquad  \int_0^\infty \norm{ \varphi(tA)x }^2   
     \,\tfrac{dt}t \le M^2   \norm{x}^2  
   \end{equation}
   for all $\varphi \in H_0^\infty(\SECT{\theta})$. In the same way we
   speak of lower square function estimates on $\SECT{\theta}$ if one
   has
\begin{equation}\label{eq:abstract-lower-eq-f-est}
\forall x \in X:\qquad   \norm{x}^2  
\le K^2 \int_0^\infty \norm{ \varphi(tA)x }^2   \,\tfrac{dt}t
\end{equation}
If $X = H$ is a Hilbert space, then upper square function estimates
for $A$ and for $A^*$ for $H_0^\infty(\SECT{\mu})$ functions for all
$\mu>\om$ are equivalent to a bounded $H_0^\infty(\SECT{\mu})$
functional calculus for all $\mu>\om$.  Moreover, by an approximate
identity argument and a duality estimate, lower square function
estimates for $A$ follow from upper estimate of its adjoint $A^*$.  We
will go into some details on the Hilbert space theory of the functional
calculus in the last section and refer at this point to
\cite{McIntosh:H-infty-calc,CowlingMcIntoshDoustYagi} for more
details.
\bigskip
Before stating our first result of this section we discuss the
following estimate for $A$
\begin{equation}\label{eq:concrete-lower-sq-f-est}\tag{$SQ_\beta$}
       \norm{x}^2  
 \le K^2 \int_0^\infty \norm{ (tA)^{-\beta} (T(2t^{2\beta})-T(t^{2\beta})) x }^2 
  \,\tfrac{dt}t
\end{equation}
that we will need to formulate the theorem. Here, $K$ is a positive
constant and $\beta \in (0,1)$. In case $\beta=\onehalf$, letting
$\varphi(z) := z^{-\onehalf} (e^{-2z} - e^{-z} )$, this corresponds to
a lower square function estimate (\ref{eq:abstract-lower-eq-f-est})
for $\varphi\in H_0^\infty(\CSECT{\pihalbe})$. As mentioned above,
(\ref{eq:abstract-lower-eq-f-est}) follows from $H^\infty-$functional
calculus when $X$ is a Hilbert space.  We will discuss again this in
the next section.
Assume that $-A$ is the generator of bounded strongly continuous
semigroup on $X$ and is injective. If we let $\varphi(z) := z^{-\beta}
(e^{-2z} - e^{-z} )$, then (\ref{eq:concrete-lower-sq-f-est}) can be
seen as
\begin{align*}
     \norm{x}^2  
& \le \; K^2 \int_0^\infty  \bignorm{  (tA)^{-\beta} 
    (T(2t^{2\beta})-T(t^{2\beta})) x }^2\, \frac{dt}t\\
&= \; K^2 \int_0^\infty \bignorm{  t^{2\beta^2-\beta} 
    \varphi(t^{2\beta} A)  x }^2\,  \tfrac{dt}t \\
(\text{letting } s=t^{2\beta}) &= \;\frac{K^2}{2\beta}  \int_0^\infty  \bignorm{
      s^{-(\onehalf -\beta)} \varphi(s A) x }^2\, \frac{ds}s,
\end{align*}
i.e. as 'weighted' lower square function estimate
for $\varphi \in H_0^\infty(\overline{\SECT{\pihalbe}})$.
 
By \cite[Theorem 6.4.6]{Haase:Buch}, the completion
$X_{\onehalf -\beta, 2}$ of $X$ with respect to the seminorm
\[
[x]_\beta =   \left( \int_0^\infty
  \bignorm{ s^{-(\onehalf -\beta)} \varphi(s A) x }^2\,
  \frac{ds}s\right)^{\onehalf}
\]
is independent of the choice of
$\varphi$ and coincides (with equivalent norms) with the real
interpolation space:
 \begin{equation}\label{eq:realinterp}
 \left(\dot{X}_{-1}(A), \dot{X}_1(A) \right)_{\frac{3}{4}
   -\frac{\beta}{2}, 2} 
 = X_{\onehalf -\beta, 2}.
 \end{equation}
 Here $\dot{X}_{-1} (A)$ is the completion of $\RANGE(A)$ with respect
 to $\norm{A^{-1}x}$ and $\dot{X}_1(A)$ is the completion of
 $\DOMAIN(A)$ with respect to $\norm{A x}$.  From
 (\ref{eq:realinterp}), it follows that
 (\ref{eq:concrete-lower-sq-f-est}) is equivalent to the continuous
 embedding
 \begin{equation}\label{con-emb}
 \left(\dot{X}_{-1}(A), \dot{X}_1(A) \right)_{\frac{3}{4}
   -\frac{\beta}{2}, 2} 
  \emb   X.
\end{equation}
For the rest of this discussion, we assume for simplicity that $A$ is
invertible. In this case, $\dot{X}_1(A) = \DOMAIN(A)$ and 
$\dot{X}_{-1} (A) = X_{-1}$.
It is a known fact that the semigroup $(T(t))$ extends to a strongly
continuous semigroup $(T_{-1}(t))$ on the extrapolation space
$X_{-1}$, whose (negative) generator $A_{-1}$ is an extension of $A$
(see \cite{EngelNagel} Chapter II, Section 5). In addition $A$ is the part of $A_{-1}$ on $X$ and hence
$\DOMAIN(A_{-1}^2) = \DOMAIN(A)$ with equivalent norms. Indeed,
\begin{align*}
x \in \DOMAIN(A_{-1}^2) 
& \Leftrightarrow x \in \DOMAIN(A_{-1}) = X, \,  A_{-1} x \in \DOMAIN(A_{-1}) \\
& \Leftrightarrow   x \in X, \, A_{-1} x \in \DOMAIN(A_{-1})\\
& \Leftrightarrow   x \in  \DOMAIN(A).
\end{align*}
The fact that $A_{-1}: X \to X_{-1}$ is an isometry implies that 
$\norm{Ax}_X = \norm{A_{-1}^2 x }_{X_{-1}}$. \\
Assume now that $\beta < \frac{1}{2}$.  We have 
 \begin{align*}
      \left( X_{-1}, \DOMAIN(A) \right)_{\frac{3}{4} -\frac{\beta}{2}, 2} 
&=    \left( X_{-1}, \DOMAIN(A_{-1}^2) \right)_{\frac{3}{4} -\frac{\beta}{2}, 2}\\
&=  \left( \DOMAIN(A_{-1}), \DOMAIN(A_{-1}^2) \right)_{\frac{1}{2} -\beta, 2}\\
&=  \left( X, \DOMAIN(A) \right)_{\frac{1}{2} -\beta, 2} \emb X. 
\end{align*}
Note that the second equality follows from
\cite[p. 105]{Triebel:interpolation}.  Hence for $\beta < \frac{1}{2}$
\begin{equation}\label{beta-petit}
  X_{\onehalf -\beta, 2} 
= \left( X_{-1}, \DOMAIN(A) \right)_{\frac{3}{4} -\frac{\beta}{2}, 2} \emb X,
\end{equation}
which means that (\ref{eq:concrete-lower-sq-f-est}) always holds for
$\beta < \frac{1}{2}$.
  
Let us finally mention that for $\beta > \frac{1}{2}$,
(\ref{eq:concrete-lower-sq-f-est}) never holds for non-negative
self-adjoint operators with compact resolvent in infinite dimension
separable Hilbert spaces. Indeed, consider such an operator $A $.  The
spectrum is discrete $\sigma(A) = \{ \lambda_n \}$ with $\lambda_n \to
+ \infty$. Applying (\ref{eq:concrete-lower-sq-f-est}) to a normalised
eigenvector $x = \varphi_n$ (associated with $\lambda_n$) yields
 \begin{align*}
  1 
&\le \; K^2 \int_0^\infty  \bignorm{  (tA)^{-\beta}\varphi_n}^2 
    (e^{-2t^{2\beta}}-e^{-t^{2\beta}})^2\, \frac{dt}t\\
&=  \; \frac{K^2}{\lambda_n^{2\beta}}  \int_0^\infty  t^{-2 \beta} 
    (e^{-2t^{2\beta}\lambda_n}-e^{-t^{2\beta} \lambda_n})^2\, \frac{dt}t\\
&=  \; \frac{K^2 \lambda_n}{2 \beta \lambda_n^{2\beta}}  \int_0^\infty   
    (e^{-2s}-e^{-s})^2\, \frac{ds}{s^2}\\
&= \; C \lambda_n^{1- 2 \beta}.
  \end{align*}
  For $\beta > \frac{1}{2}$, the last term goes to $0$ as $n$ tends to
  $+\infty$. This shows that (\ref{eq:concrete-lower-sq-f-est}) cannot
  hold.
  
  Now we come to our main result of this section concerning exact
  observability.
\begin{theorem}\label{thm:main-bernhard}
  Let $-A$ be the generator of a bounded semigroup on the Banach space
  $X$ and assume that $A$ is injective and has dense range. Let $C:
  \DOMAIN(A) \to Y$ be bounded and suppose that there exists $\beta
  \in (0, 1)$ such that the lower square function estimate
  (\ref{eq:concrete-lower-sq-f-est}) and (\ref{eq:Rbeta}) are
  satisfied.  Then there exists a constant $m>0$ such that 
\begin{equation}\label{eq:exobs-on-core}
  m^2 \norm{x}^2 \le 
                 \int_0^\infty \norm{ C T(t)x }_Y^2 \,dt
\end{equation}
for all $x\in \DOMAIN(A) \cap \RANGE(A)$.
\end{theorem}
\begin{proof}
  Fix $x\in \DOMAIN(A) \cap \RANGE(A)$ and apply
  the lower square function estimate
  (\ref{eq:concrete-lower-sq-f-est}) to obtain
\begin{align*}
  \norm{x}^2 
& \le K^2 \int_0^\infty \norm{ (tA)^{-\beta} (T(2t^{2\beta})-T(t^{2\beta})) x }^2 
   \,\tfrac{dt}t\\
& \le \tfrac{K^2} {\delta^2} \int_0^\infty \norm{ C A^{-(1-\beta)}
     (tA)^{-\beta} (T(2t^{2\beta})-T(t^{2\beta})) x }_Y^2 \, \tfrac{dt}t\\
& = \tfrac{K^2} {\delta^2} \int_0^\infty \norm{ C A^{-1} 
     (T(2t^{2\beta})-T(t^{2\beta}))  x }_Y^2 \, \tfrac{dt}{t^{1+2\beta}}.
\end{align*}
Using the fact that $x\in \DOMAIN(A)$, we can write
\[
   C A^{-1} (T(2t^{2\beta})-T(t^{2\beta}))  x = - C
   A^{-1}\int_{t^{2\beta}}^{2t^{2\beta}} A T(s)x \,ds = -
   \int_{t^{2\beta}}^{2t^{2\beta}} C T(s)x \,ds.
\]
Using this and the previous estimates yields
\begin{align*}
\norm{x}^2 
& \le  \tfrac{K^2} {\delta^2} \int_0^\infty \norm{ \int_{t^{2\beta}}^{2t^{2\beta}}
     C T(s)x \,ds }_Y^2 \, \tfrac{dt}{t^{1+2\beta}}\\
&\le \tfrac{K^2} {\delta^2} \int_0^\infty
      \biggl(\int_{t^{2\beta}}^{2t^{2\beta}}  1\cdot \norm{C T(s)x}_Y \,ds
       \biggr)^2 \, \tfrac{dt}{t^{1+2\beta}}\\
& \le \tfrac{K^2} {\delta^2} \int_0^\infty
     \int_{t^{2\beta}}^{2t^{2\beta}} \norm{ C T(s)x }_Y^2 \,ds\, \tfrac{dt}{t}\\
& = \tfrac{K^2} {\delta^2} \int_1^2 \int_0^\infty  \norm{ C
        T(t^{2\beta} u)x }_Y^2 t^{2\beta-1} \,dt  \, du\\
& = \log(2) \tfrac{ K^2} {2\beta\delta^2} \int_0^\infty  
       \norm{ C T(r)x}_Y^2 \,dr. 
\end{align*}
This shows (\ref{eq:exobs-on-core}) with $m = \sqrt{\log(2)/2\beta}
\frac{ K} {\delta}$.
\end{proof}
Notice that if $0 \in \varrho(A)$ then $\RANGE(A) = X$ and hence
$\DOMAIN(A) \cap \RANGE(A) = \DOMAIN(A)$. In this case
(\ref{eq:exobs-on-core}) holds for all $x \in \DOMAIN(A)$ and this
means that $C$ is exactly observable for $A$.
\begin{corollary}\label{cor:main-thm-A-inv}
  Let $-A$ be the generator of a bounded semigroup on the Banach space
  $X$ and assume that $A$ is boundedly invertible.  Let $C: \DOMAIN(A)
  \to Y$ be a bounded operator and assume that there exists $\beta \in
  (0, 1)$ such that the lower square function estimate
  (\ref{eq:concrete-lower-sq-f-est}) and (\ref{eq:Rbeta}) are
  satisfied. Then $C$ is exactly observable for $A$.
\end{corollary}
\begin{corollary}\label{cor:main-thm-admiss}
  Let $-A$ be the generator of a bounded semigroup on the Banach space
  $X$ and assume that $A$ is injective and has dense range. Let $C:
  \DOMAIN(A) \to Y$ be infinite-time admissible for $A$ and assume
  that there exists $\beta \in (0, 1)$ such that the lower square
  function estimate (\ref{eq:concrete-lower-sq-f-est}) and
  (\ref{eq:Rbeta}) are satisfied. Then $C$ is exactly observable for $A$.
\end{corollary}
\begin{proof}
It remains now to extend the estimate (\ref{eq:exobs-on-core}) from
Theorem~\ref{thm:main-bernhard} for all $x \in \DOMAIN(A)$. This is an
easy task.  For $x \in \DOMAIN(A)$ the sequence
\[
   x_n := n(n{+}A)^{-1} x - n^{-1} (n^{-1}{+}A)^{-1}x 
    \in \DOMAIN(A) \cap \RANGE(A)
\] 
converges to $x$ in $\DOMAIN(A)$ (for the graph norm). Therefore,
$Cx_n$ converges in $Y$ to $Cx$ and $\int_0^\infty \norm{ C T(t)x_n
}_Y^2 \,dt$ converges to $\int_0^\infty \norm{ C T(t)x }_Y^2 \,dt$
since $C$ is supposed to be infinite-time admissible. We obtain
(\ref{eq:exobs-on-core}) for all $x \in \DOMAIN(A)$.  
\end{proof}
In the next corollary we obtain a criterion for finite time exact
observability. 
\begin{corollary}\label{cor:ex-tau}
  Let $-A$ be the generator of a bounded semigroup on the Banach space
  $X$ and assume that there exists a constant $\om > 0$ such that 
  $A + \om$ satisfies 
  (\ref{eq:concrete-lower-sq-f-est}). Assume that 
  $C: \DOMAIN(A) \to Y$ is bounded and that
\[
   \norm{C (\om+A)^{-(1-\beta)} x}_Y \ge \delta  \norm{x}
\]
for $x \in \DOMAIN(A)$.  Then $C$ is exactly observable in finite time.
\end{corollary}
\begin{proof} 
  We apply the previous theorem to $\om + A$ and obtain
\[ 
  \norm{x}^2 \le M \int_0^\infty \norm{ C e^{-\om t} T(t)x }_Y^2 \,dt
\]
for all $x \in \DOMAIN(A)$.  We split the right hand side into two
parts and write
\begin{align*}
    \int_0^\infty \norm{ C e^{-\om t} T(t)x }_Y^2 \,dt 
& = \int_0^\tau  \norm{ C e^{-\om t} T(t)x }_Y^2 \,dt+
    \int_\tau^\infty \norm{ C e^{-\om t} T(t)x }_Y^2 \,dt\\
&\le\int_0^\tau  \norm{ C T(t)x }_Y^2 \,dt + 
     \int_0^\infty \norm{ C e^{-\om(\tau + t)} T(t + \tau)x }_Y^2 \,dt\\
& =  \int_0^\tau  \norm{ C T(t)x }_Y^2 \,dt +  
     e^{-\om \tau} \int_0^\infty \norm{ C e^{-\om t} T(t )T(\tau)x }_Y^2 \,dt.
 \end{align*}
 Since $C$ is infinite-time admissible for $e^{-\om t} T(t)$ and the
 semigroup $T(t)$ is bounded we have for some constants $M', M''$
\[ 
   \int_0^\infty \norm{ C e^{-\om t} T(t) T( \tau)x }_Y^2 \,dt \le M'
   \norm{T(\tau) x}^2 \le M'' \norm{x}^2.
\]
Therefore,
\[ 
\norm{x}^2 \le M  \int_0^\tau  \norm{ C T(t)x }_Y^2 \,dt + MM''
e^{-\om \tau} \norm{x}^2.
\]
If we choose $\tau$ large enough such that $MM'' e^{-\om \tau} < 1$ we
obtain the desired inequality.
\end{proof}
As explained at the beginning of this section,
(\ref{eq:concrete-lower-sq-f-est}) holds for all $\beta < \frac{1}{2}$
and all generators of bounded semigroup (see (\ref{beta-petit})).
Therefore, applying Theorem~\ref{thm:main-bernhard} with $\beta =
\frac{1}{2} - \eps$, we obtain the following corollary.
\begin{corollary}
  Let $-A$ be the generator of a bounded semigroup $(T(t))_{t\ge0}$ on
  a $X$ and assume that $A$ is invertible. Then, if
  $\norm{CA^{-\onehalf + \eps} x} \ge \delta \norm{x}$ for some $\eps,
  \delta>0$ and all $x\in \DOMAIN(A)$, $C$ is infinite-time exactly
  observable for $A$.
\end{corollary}
\section{Exact observability on Hilbert spaces}
\begin{proposition}\label{prop:necessarily-contraction-sg}
  Let $X$ and $Y$ be Hilbert spaces. Then, if $(A, C)$ is exactly
  observable and admissible in infinite time, $T(t)_{t\ge 0}$ is
  similar to a contraction semigroup.
\end{proposition}
\begin{proof}
  Denote by $\langle x, y \rangle_Y$ the scalar product of $Y$ and
  define for $x, y \in \DOMAIN(A) $
\[
\langle x, y \rangle_X^\sim
  := \int_0^\infty \langle CT(t)x, CT(t)y \rangle_Y\,dt.
\]
This is clearly a bilinear (or sesquilinear) form on $\DOMAIN(A)
\times \DOMAIN(A)$.  Admissibility and exact observability imply that
$\norm{x}_X $ and $\norm{x}_{X}^\sim$ are equivalent.  By density, we
extend this to all $ x \in X$ and $\norm{\cdot}_X^\sim$ is associated
with a scalar product on $X$.  With respect to the new norm,
\[
   \norm{ T(t)x }_{X}^\sim 
= \Bigl(\int_0^\infty \norm{ CT(t{+}s)x }^2\,ds \Bigr)^\onehalf 
= \Bigl(\int_t^\infty \norm{ CT(s)x }^2\,ds \Bigr)^\onehalf 
\le     \norm{ x }_{X}^\sim 
\]
and so  $T(t)_{t\ge 0}$ is a contraction semigroup with respect to 
$\norm{\cdot}_X^ \sim$.
\end{proof}
Now we turn back to Theorem~\ref{thm:main-bernhard}. As mentioned at
the beginning of the previous section, the lower square function
estimate (\ref{eq:concrete-lower-sq-f-est}) holds for small $\beta$
for all generators of bounded strongly continuous semigroups. It is
then tempting to use the theorem for small $\beta$ in order to include
a large class of semigroups $T(t)$.  On the other hand, if we assume
that $0\in \varrho(A)$ and $\norm{CA^{-\al-\eps} x} \ge \delta
\norm{x}$, one also has
\[
    \norm{ CA^{-\al} x } 
=   \norm{ CA^{-\al-\eps} A^{\eps} x } 
\ge \delta \norm{A^{\eps} x }
\ge \delta' \norm{x}. 
\]
That is, the invertibility condition on $C A^{-(1-\beta)}$ becomes
more restrictive when $\beta$ decreases. To admit more observation
operators $C$ one therefore seeks for values of $\beta$ large enough.
Combining both conditions forces to play with different values of
$\beta$ in different situations. In the following corollary we choose
$\beta=\onehalf$.
\begin{corollary}\label{cor:contraction-sg}
  Let $-A$ be the generator of a semigroup of contractions
  $(T(t))_{t\ge0}$ on a Hilbert space $X$. If $A$ has dense range and
  $\norm{CA^{-\onehalf} x} \ge \delta \norm{x}$ for all $x\in
  \DOMAIN(A) \cap \RANGE(A)$, then
  \begin{equation}\label{eq:exobs-on-core-h}
  m^2 \norm{x}^2 \le 
                 \int_0^\infty \norm{ C T(t)x }_Y^2 \,dt
\end{equation}
for all $x\in \DOMAIN(A) \cap \RANGE(A)$. In addition, if either 
       $A$ is invertible 
 or $C$ is infinite-time admissible for $A$
  then $C$ is
  infinite-time exactly observable for $A$.
\end{corollary}
Notice that in view of
Proposition~\ref{prop:necessarily-contraction-sg}, the hypothesis of a
semigroup of contractions is necessity to be able to conclude in the
case that $C$ is admissible.
\begin{proof}
  By the Lumer-Phillips theorem, $A$ is an accretive operator, i.e.
  $\Re \langle A x, x \rangle \ge 0$ for all $x\in \DOMAIN(A)$. Since
  $A$ has dense range and $X$ is reflexive, $A$ is actually injective
  (cf.~\cite[Theorem.~3.8]{CowlingMcIntoshDoustYagi}).  We need to
  verify that $A$ admits lower square function estimates
  (\ref{eq:abstract-lower-eq-f-est}) with functions $\varphi$ that are
  bounded holomorphic on $\SECT{\pihalbe}$ and continuous on
  $\overline{\SECT{\pihalbe}}$.  We shall explain how this follows
  from the functional calculus.
  First notice that $A$ has a bounded
  $H^\infty(\overline{\SECT{\pihalbe}})$ functional calculus. This
  means that for every bounded holomorphic function $\varphi$ on
  $\SECT{\pihalbe}$ and continuous on $\overline{\SECT{\pihalbe}}$,
  $\varphi(A)$ is well defined and
\[
   \norm{ \varphi(A) }_{{\mathcal L}(X)} \le M \sup_{z \in
     \overline{\SECT{\pihalbe}}} | \varphi(z) |.
\]
This is essentially von Neumann's inequality for contractions on a
Hilbert space. Indeed, if $A$ is accretive, $T := (A-1)(A+1)^{-1}$ is
a contraction and von Neumann's inequality states $\norm{ p(T) } \le
\norm{p}_{H^\infty(\DD)}$ for every polynomial $p$. From this, the
$H^\infty$-calculus can be derived by approximation arguments. Two
different direct proofs for the boundedness of the
$H^\infty(\overline{\SECT{\pihalbe}})$ calculus are given in
\cite[Theorem~7.1.7]{Haase:Buch}. One uses a dilation theorem of the
semigroup into a unitary $C_0$-group due to Sz.-Nagy. The second
exploits accretivity of $A$ from a 'numerical range' viewpoint and can
be seen as the most simple case of the Crouzeix-Delyon theorems
\cite{CrouzeixDelyon,BadeaCrouzeixDelyon}.
Having the boundedness of the functional calculus on
$\overline{\SECT{\pihalbe}}$ in hands we certainly have upper square
function estimates for functions in $H_0^\infty(\SECT{\theta})$ when
$\theta>\pihalbe$. This is well known and proved by McIntosh
\cite{McIntosh:H-infty-calc}.  For the particular functions
$\psi_\al(z) := z^{\al} / (1+z)$ for $\al\in (0,1)$ this means that
for some positive constants $k_\al$
\begin{equation}\label{eq:calcul}
k_\al \int_0^\infty \norm{ \psi_\al(tA)x }^2 
   \,\tfrac{dt}t \le \norm{x}^2.
   \end{equation}
   Given now a function $\psi \in H_0^\infty(\overline{S_\pihalbe})$,
   we choose $\eps>0$ small such that $|\psi(z)| \le
   M\max(|z|^{2\eps}, |z|^{-2\eps})$ and write
\[
       \psi(z) = \psi_\eps(z) \times z^{-\eps}\psi(z) +
       \psi_{1-\eps}(z) \times z^{\eps}\psi(z).
\]
Notice that $z^{\pm \eps}\psi(z) \in
H_0^\infty(\overline{S_\pihalbe})$.  Therefore, upper square function
estimate for $A$ with the function $\psi$ follow from
(\ref{eq:calcul}) using the boundedness of the
$H^\infty(\overline{S_\pihalbe})$ calculus.
All what we explain here works also for the adjoint $A^*$. By a
duality argument as in \cite{McIntosh:H-infty-calc} or
\cite{CowlingMcIntoshDoustYagi}, we pass from upper square function
estimates for $A^*$ to lower square function estimates for $A$. As a
particular case, (\ref{eq:concrete-lower-sq-f-est}) holds with $\beta
= \frac{1}{2}$.  We then apply Theorem~\ref{thm:main-bernhard}.
\end{proof}
Again as in Corollary~\ref{cor:ex-tau}, we can obtain observability in
finite time by adding a constant $w$ to $A$. That is
\begin{corollary}\label{cor:ex-tau-hilbert}
  Let $-A$ be the generator of a semigroup $(T(t))$ on the Hilbert
  space $X$. Suppose that there exists a constant $\om > 0$ such that
  $A + \om$ is accretive and that $C: \DOMAIN(A) \to Y$ is bounded.  
  If $\norm{C(A+\om)^{-\onehalf} x} \ge \delta \norm{x}$
  for all $x\in \DOMAIN(A)$ and some $\delta >0$, then $C$ is exactly
  observable in finite time.
\end{corollary}
\begin{examples}
1- Consider the  Schr\"odinger equation
\[
\left\{
  \begin{array}{lcl}
    \frac{\partial u}{\partial t} &=& i \Delta u  \\
    u &=& 0\  {\rm on }\,  \partial \Om\\
    y(t) &=& C u(t) = \nabla u(t)
  \end{array}
\right.
\]
with $X = L^2(\Om)$, $Y = (L^2(\Om))^d$ and $\Om$ is any open subset
of $\RR^d$ with boundary $\partial \Om$. In this problem, $\Delta$
denotes the Laplacian with Dirichlet boundary conditions. It is
(the negative of) the associated operator with the symmetric form
\[
\mathfrak{a}(u,v) := \int_\Om \nabla u \nabla v dx, \ \,
\DOMAIN(\mathfrak{a}) = W_0^{1,2}(\Om).
\]
Since 
\[
\norm{(-i \Delta)^{1/2} u}_2^2 = \norm{(-\Delta)^{1/2} u}_2^2  =
\norm{ Cu}_2^2 
\]
we may conclude from Corollary~\ref{cor:contraction-sg} that 
\begin{equation}\label{schr}
  \delta \norm{f}_2^2 \le 
                 \int_0^\infty \norm{ \nabla  e^{it\Delta} f }_2^2 \,dt
\end{equation}
for all $f\in \DOMAIN(\Delta) \cap \RANGE(\Delta)$.  Note that we can
replace here the Dirichlet boundary condition by the Neumann one. The
arguments are the same, we just need to replace $
\DOMAIN(\mathfrak{a}) = W_0^{1,2}(\Om)$ by $ \DOMAIN(\mathfrak{a})
=W^{1,2}(\Om)$. The same method applies for more general boundary
conditions.
\vspace{.5cm}
2- Let  $A$ be the uniformly elliptic operator 
\[
A = - \sum_{j,k=1}^d
\frac{\partial}{\partial x_k} \bigl( a_{jk} \frac{\partial}{\partial
  x_j} \bigr)
\] with bounded measurable coefficients $a_{jk} \in
L^\infty(\RR^d)$. The operator $A$ is defined by sesquilinear form
techniques (see for example \cite{Ouhabaz:book}) and note that
$A$ is not necessarily self-adjoint. \\
It is a standard fact that $-A$ generates a contraction semigroup on
$L^2(\RR^d)$. Consider the problem
\[
\left\{
  \begin{array}{lcl}
    \frac{\partial u}{\partial t} &=& -A u  \\
    y(t) &=& C u(t) = \nabla u(t)
  \end{array}
\right.
\]
with $X = L^2(\RR^d)$ and $Y = (L^2(\RR^d)^d$. By the solution of the
Kato's square root problem (see
\cite{Auscher-Hofmann-Lacey-McIntosh-Tchamitchian}), it is known that
\begin{equation}\label{kato}
\norm{\nabla u}_2 \approx  \norm{A^\onehalf} u \, \, \forall u \in W^{1,2}(\RR^d).
\end{equation}
On the other hand, the accretivity of $A$ implies an $H^\infty$
functional calculus on $L^2(\RR^d)$. Hence it satisfies upper and
lower square function estimate (\ref{eq:abstract-upper-eq-f-est}) and
(\ref{eq:abstract-lower-eq-f-est}). In particular,
\begin{equation}\label{sq2}
\norm{f}_2^2 \approx \int_0^\infty \norm{ A^\onehalf  e^{-tA} f }_2^2 \,dt.
\end{equation}
This  implies that $C = \nabla$ is both admissible and exactly 
observable for $A$ in infinite time. \\
Note that the semigroup $e^{-tA}$ is bounded holomorphic on some
sector $\SECT{\omega}$ and $e^{-te^{iw} A}$ is a contraction on
$L^2(\RR^d)$ (see \cite{Ouhabaz:book}). Taking the maximal angle $w$,
we obtain a contraction semigroup $e^{-te^{iw} A}$ which is not
holomorphic. Since $(e^{iw}A)^\onehalf = e^{iw/2}A^\onehalf $ we see
that (\ref{kato}) holds with $e^{iw}A$ in place of $A$. By Corollary
\ref{cor:contraction-sg} we have
\begin{equation}\label{chal}
  \delta \norm{f}_2^2 \le 
                 \int_0^\infty \norm{ \nabla  e^{-t e^{iw}A} f }_2^2 \,dt
\end{equation}
for $f \in \DOMAIN(A) \cap \RANGE(A)$. We may consider the same
problem on a bounded Lipschitz domain instead of $\RR^d$. In this
case, $A$ is invertible. Hence $e^{iw} A$ is also invertible and we
obtain (\ref{chal}) for all $f \in \DOMAIN(A)$. This means that $C$ is
exactly observable for $e^{iw}A$.
\end{examples}
\def\cprime{$'$}
\providecommand{\bysame}{\leavevmode\hbox to3em{\hrulefill}\thinspace}

\end{document}